\documentclass[]{interact}
\usepackage{dsfont}
\usepackage[caption=false]{subfig}

\usepackage[numbers,sort&compress]{natbib}
\bibpunct[, ]{[}{]}{,}{n}{,}{,}
\renewcommand\bibfont{\fontsize{10}{12}\selectfont}
 
\theoremstyle{plain}
\newtheorem{theorem}{Theorem}[section]
\newtheorem{lemma}[theorem]{Lemma}
\newtheorem{corollary}[theorem]{Corollary}

\newtheorem{assumption}[theorem]{Assumption}

\theoremstyle{definition}

\theoremstyle{remark}
\newtheorem{remark}{Remark}

\usepackage{graphicx}
\usepackage[dvipsnames]{xcolor}

\usepackage{algorithm}
\usepackage{algorithmic}
\usepackage{caption}

\usepackage{makecell}
\usepackage{mathtools}
\usepackage{amsfonts}
\usepackage{amssymb}
\usepackage{url}
\usepackage{dirtytalk}
\usepackage{tikz}
\usepackage{nicefrac}
\usepackage{hyperref}
\usepackage{multirow}
\usepackage{cancel}

\newcommand*{\R}{{\mathbb R}}

\def\bx{\mathbf x}
\def\by{\mathbf y}
\def\bz{\mathbf z}
\def\mW{\mathbf W}

\def\mI{\mathbf I}

\def\cG{\mathcal G}
\def\cE{\mathcal E}
\def\cV{\mathcal V}
\def\cL{\mathcal L}

\def\eps{\varepsilon}

\DeclareMathOperator*{\argmin}{argmin}
\DeclareMathOperator*{\argmax}{argmax}

\DeclareMathOperator{\col}{col}

\DeclareMathOperator{\kernel}{Ker}

\newcommand{\norm}[1]{\left\| #1 \right\|}
\newcommand{\normtwo}[1]{\left\| #1 \right\|_2}
\newcommand{\angles}[1]{\left\langle #1 \right\rangle}

\newcommand{\ds}{\displaystyle}

\begin{document}

\articletype{ARTICLE TEMPLATE}

\title{The Mirror-Prox Sliding Method for Non-smooth decentralized saddle-point problems \footnote{The work was supported by a grant for research centers in the field of artificial intelligence, provided by the Analytical Center for the Government of the Russian Federation in accordance with the subsidy agreement (agreement identifier 000000D730321P5Q0002) and the agreement with the Ivannikov Institute for System Programming of the Russian Academy of Sciences dated November 2, 2021 No. 70-2021-00142.
}
}

\author{
\name{Ilya Kuruzov\textsuperscript{a,b}\thanks{CONTACT Ilya Kuruzov. Email: kuruzov.ia@phystech.edu } and Alexander Rogozin\textsuperscript{a} and Demyan Yarmoshik\textsuperscript{a} and Alexander Gasnikov\textsuperscript{a,c,d}}
\affil{
\textsuperscript{a}Moscow Institute of Physics and Technology, Dolgoprudny, Russia;
\textsuperscript{b}Institute for Information Transmission Problems, Moscow, Russia;
\textsuperscript{c}Skolkovo Institute of Science and Technology, Moscow, Russia;
\textsuperscript{d}ISP RAS Research Center for Trusted Artificial Intelligence
}
}

\maketitle

\begin{abstract}
The saddle-point optimization problems have a lot of practical applications. This paper focuses on such non-smooth convex-concave problems in decentralized case. This work is based on recently proposed sliding for centralized variational inequalities. Through 
specific penalization method and this sliding we obtain algorithm for non-smooth decentralized saddle-point problems. Note, the proposed method approaches lower bounds both for number of communication rounds and calls of (sub-)gradient per node. As a helpful result, we obtain the convergence of sliding for some non-smooth variational inequalities.
\end{abstract}


\section{Introduction}
In this paper we study decentralized non-smooth saddle point problems (SPP). Practical importance of the non-smooth SPP setting is explained by its applications, which include Total Variation image processing problems \cite{drori2015simple} and Wasserstein Barycenter problem, see e.g. \cite{yufereva2022decentralized}.
In turn, decentralization is a popular and effective \cite{lian2017can} approach for handling large-scale problems.

To date the theory of decentralized optimization and decentralized saddle point problems can be considered quite well established. Algorithms that match lower bounds are known for almost all of these problems subclasses \cite{gorbunov2022recent}. So, $\mu$-strongly convex case with $L$-smooth functional was considered in work \cite{opt_alg_beznos}. The authors proved lower bounds for this case and obtained optimal algorithm in deterministic case.  The case of non-strongly convex smooth case was researched in work \cite{opt_alg_l_smooth}. Authors proposed an algorithm based on Extra-Step method  and proved its optimality in this case with respect to oracle calls.


At the same time, the case of non-smooth SPP was left uncovered by recent research. Note, that saddle-point problems arise in different applications, e.g. fair training \cite{Baraz_non_convex_minimax}, adversarial machine learning \cite{Goodfellow_scale}, Generative adversarial network \cite{Goodfellow}. Some of such problems also include non-smooth object functions. In particular, it is problems with $\ell_1$-penalization \cite{sparse_gan}, \cite{sparse_dictionary} and some statements of Support Vector Machine \cite{svm_leon}. The significant part of such problems should be solved over several devices. So, our goal is to close this gap
by providing an algorithm for decentralized non-smooth SPP, which is optimal in terms of oracle calls and gossip steps.


To obtain such optimal algorithm, we use penalization for initial problem to avoid consensus constraint. This idea was developed in \cite{dvinskikh2019decentralized,gorbunov2019optimal}. It allows considering non-smooth problem without additional constraints. Further, it can be presented as variational inequality. In \cite{gorbunov2022recent} for non-smooth (not saddle point) optimization this idea was used with sliding from work \cite{lan2020first}. So, in this work we want to apply this penalization with another sliding procedure proposed in \cite{Lan_MPS}. 

Note, that the both parts of operator in variational inequality should be smooth for the original method from \cite{Lan_MPS}. To apply sliding to non-smooth problems we consider the condition similar to relative smoothness. The main problem of such approach is to prove that the error is not accumulated. We generalize the original results for new condition in this work and unite them with ideas of penalization.


{\textbf{Problem statement}. In this paper we consider a saddle point problem
\begin{equation*}
    \begin{aligned}
	    \min_{x \in X} \max_{y \in Y} \left[f(x, y)=\frac{1}{m} \sum_{i=1}^m f_i(x, y)\right], 
	\end{aligned}
\end{equation*}
where $X, Y$ are closed convex bounded sets and $f_i(x_i, y_i)$ are convex-concave non-smooth functions. Each $f_i$ is stored at a separate computational agent. The agents do not have a central aggregating server, but they can communicate to each other via a communication network. The network is represented by a connected undirected graph $\cG = (\cV, \cE)$, and the agents can exchange information if and only if the corresponding nodes are connected by an edge.
}

{
\textbf{Notations.} By $\|\cdot\|$ we mean the $\ell_2$-norm, i.e. $\|x\|=\sqrt{\langle x, x\rangle}.$ For a given convex-concave function $F(x,y)$ we let $\nabla_x F(x,y)$ be some subgradient of function $F(\cdot, y)$ at point $x$ and $\nabla_y F(x,y)$ be some subgradient of function $-F(x, \cdot)$ multiplied by $-1$. We also let $\otimes$ denote the Kronecker product. For vectors $x_1, \ldots, x_m\in\R^d$ we define $\col(x_1, \ldots, x_m) = [x_1^\top\ldots x_m^\top]^\top$ the column-stacked vector consisting of them. We introduce a prox-function $d(x)$ that is $1$-strongly convex w.r.t. $\norm{\cdot}$, and a Bregman divergence $V(x, y) = d(x) - d(y) - \angles{\nabla d(y), x - y}$ (see \cite{bregman_orig}).
}

{
\textbf{Paper organization}.
In Section \ref{sec:MPS} we generalize theoretical results for Mirror-Prox Sliding (MPS) method for the non-smooth unconstrained SPP in deterministic and stochastic setups. This is done by adding inexactness to the analysis of MPS method.
In Section \ref{sec:decentr_spp} we reformulate the constrained SPP as an unconstrained SPP via penalization.
The properties of the reformulation are analyzed in Section \ref{sec:spp_penalty}.
In Section \ref{sec:main_result} we formulate the main result about convergence of sliding for decentralized saddle-point problem. Finally, we present a discussion in Section~\ref{sec:discussion} and make concluding remakrs in Section~\ref{sec:conclusion}.
}

\section{The mirror-prox sliding method with inexactness}\label{sec:MPS}

In this section we restrict our attention to the mirror-prox sliding method with inexactness. The source of inexactness is the non-smoothness of the field (we discuss it further in the paper in Remark \ref{remark:assumption_examples}). We revisit the analysis of \cite{Lan_MPS} with a modification (additional term $\delta$) that describes the inexactness. Namely, consider the problem of the following variational inequality (VI) over convex closed set $Z$:
\begin{equation}\label{vi_problem}
    \langle H (z)+\nabla G(z), z^*-z\rangle \leq 0~~~ \forall z \in Z,
\end{equation}
where $H(z)$ is a monotone field and $G(z)$ is a differentiable convex function. 
{Formally, we impose the following assumptions.
\begin{assumption}\label{assum:field_h}
	Field $H(z)$ is monotone, i.e. for any $z_1, z_2\in Z$ it holds
	\begin{align*}
		\angles{H(z_2) - H(z_1), z_2 - z_1}\geq 0.
	\end{align*}
	Moreover, there exist constants $M > 0$ and $\delta> 0$ such that for any $z_1, z_2, z_3\in Z$ we have
	\begin{align}
		\label{delta_L_H}
		\langle H(z_1)-H(z_2), z_1-z_3\rangle \leq \frac{M}{2}\| z_1- z_2\|^2+\frac{M}{2}\|z_1-z_3\|^2+\delta.
	\end{align}
\end{assumption}
Note that in Assumption~\ref{assum:field_h} term $\delta$ measures the inexactness. This inexactness emerges from non-smoothness of $H$. We discuss sufficient conditions for the inequality above to hold below in the paper (see Remark \ref{remark:assumption_examples}).
\begin{assumption}\label{assum:function_g}
	Function $G(z)$ is convex and $L$-smooth, i.e.
	\begin{align*}
		0\leq G(z_2) - G(z_1) - \angles{\nabla G(z_1), z_2 - z_1}\leq \frac{L}{2}\norm{z_2 - z_1}_2^2.
	\end{align*}
\end{assumption}
Assumption~\ref{assum:function_g} is standard for optimization. In decentralized case, function $G$ refers to penalty terms that are added to control consensus accuracy.
}

We consider proposed in recent work \cite{Lan_MPS} the mirror-prox sliding method in deterministic (see Algorithm \ref{alg:mps_deterministic}) and stochastic cases (see Algorithm \ref{alg:mps_stochastic}). Similarly to \cite{Lan_MPS}, in order to evaluate efficiency of methods we define:
\begin{align*}
	Q(z_1, z_2) = G(z_1)-G(z_2) + \langle H(z_2), z_1-z_2\rangle.
\end{align*}
Note that $G(z)$ is convex and therefore if $Q(z_1, z_2)\leq 0$, then $z_1$ is a solution of VI~\eqref{vi_problem}. We measure solution accuracy at point $\hat z$ via dual gap: $\hat z$ is an $\eps$-accurate solution to problem~\eqref{vi_problem} if $\sup_{z\in Z}~ Q(z_1, z_2)\leq\eps$.

In the rest of this section we revisit the proofs from \cite{Lan_MPS} under Assumption~\ref{assum:field_h}. The difference with \cite{Lan_MPS} is the additional inexactness, i.e. additional term $\delta$ in \eqref{delta_L_H}. The convergence result for Algorithm~\ref{alg:mps_deterministic} (deterministic) is provided in Theorem~\ref{theorem_mps_deterministic} and the result for Algorithm~\ref{alg:mps_stochastic} (stochastic) is given in Theorem~\ref{theorem_mps_stochastic}. We begin the proof with a sequence of lemmas starting from deterministic case.

\vspace{1.5cm}
\noindent\textbf{{Deterministic case}}. 
\begin{algorithm}[H]
	\caption{The mirror-prox sliding (MPS) method for solving  VI \eqref{vi_problem}} \label{alg:mps_deterministic}
	\begin{algorithmic}[1]
		\REQUIRE Initial point $z_0 \in Z$, external steps number $N$,  sequence on internal steps number $\{T^k\}_{k=1}^n$, sequences of parameters $\{\beta_k\}_{k=1}^n, \{\eta_k^t\}_{t\leq T_k, k\leq N}, \{\gamma_k\}_{k=1}^N$
		
		\STATE $\overline{z}_0:=z_0$
		\FOR{$k =0, \ldots, N$}
		\STATE Compute $\underline{z}_k=(1-\gamma_k)\overline{z}_{k-1}+\gamma_k z_{k-1}$ and set $z_k^0 = z_{k-1}.$
		\FOR{$t =1, \ldots, T_k$}
		\STATE Compute
		$$\tilde{z}_k^t = \argmin\limits_{z\in Z} \langle \nabla G(\underline{z}_k)+H(z_k^{t-1}), z\rangle + \beta_k V(z_{k-1}, z) + \eta_k^t V(z_k^{t-1}, z),$$        
		$${z}_k^t = \argmin\limits_{z\in Z} \langle \nabla G(\underline{z}_k)+H(\tilde{z}_k^t), z\rangle + \beta_k V(z_{k-1}, z) + \eta_k^t V(z_k^{t-1}, z).$$
		\ENDFOR 
		\STATE Set $z_k:=z_k^{T_k}, \tilde{z}_k:=\frac{1}{T_k}\sum\limits_{t=1}^{T_k} \tilde{z}_k^t$ and compute $\overline{z}_k=(1-\gamma_k)\overline{z}_{k-1}+\gamma_k\tilde{z}_k.$
		\ENDFOR 
		
		\RETURN $\overline{z}_N.$
	\end{algorithmic}
\end{algorithm}

Let us analyze MPS method in the deterministic case. Firstly, Lemma 2.1 from work \cite{Lan_MPS} does not require some properties of operator $H$. So, the result is true for non-lipschitz operator too. So, we will use it in the following form:
\begin{lemma}
\label{lemma:ext_determenistic}
For any $\gamma_k \in [0, 1]$, we have
\begin{align*}
    &\left[G(\overline{z}_k)-G(z)+\langle H(z), \overline{z}_k-z\rangle\right]\\
    & -(1-\gamma_k)\left[G(\overline{z}_{k-1})-G(z)+\langle H(z), \overline{z}_{k-1}-z\rangle\right]\\
    \leq & \gamma_k\left[\langle \nabla G(\underline{z}_k)+ H(z), \tilde{z}_k-z \rangle+\frac{L\gamma_k}{2} \|\tilde{z}_k-z_{k-1}\|^2\right].
\end{align*}
\end{lemma}

Nevertheless, Lipschitz condition is important for analysis of internal iterations of Algorithm \ref{alg:mps_deterministic}. We obtain the following generalization of result for these points.
\begin{lemma}
\label{lemma:int_determenistic}
If $M\leq \beta_k+\eta_k^t, \forall k, t\geq 1$ and $\eta_k^t \leq \beta_k + \eta_k^{t-1}, \forall k\geq 1, t\geq 2,$ then
\begin{align*}
    &\langle \nabla G(\underline{z}_k)+ H(z), \tilde{z}_k-z \rangle\\
    \leq & -\beta_k V(z_{k-1}, \tilde{z}_k) +\left(\beta_k+\frac{\eta_k^1}{T_k}\right)V(z_{k-1}, z)- \frac{\beta_k+\eta_k^{T_k}}{T_k} V(z_k, z) + \delta, \forall z\in Z.
\end{align*}
\end{lemma}

\begin{proof}
    In the work \cite{Lan_MPS} authors proved the following inequality:
    \begin{align*}
        &\langle G(\underline{z}_k), \tilde{z}_k^t - z \rangle + \langle H(\tilde{z}_k^t), z_k^t - z\rangle + \langle H({z}_k^{t-1}), \tilde{z}_k^t - z\rangle\\
        \leq & \beta_k\left(V(z_{k-1}, z) - V(z_k^t, z) - V(z_{k-1}, \tilde{z}_k^t)-V(\tilde{z}_k^t, z^t_k)\right)\\
        & + \eta_k^t \left(V(z_{k}^{t-1}, z) - V(z_k^t, z) - V(z_{k}^{t-1}, \tilde{z}_k^t)-V(\tilde{z}_k^t, z^t_k)\right), \forall z \in Z.
    \end{align*}
Using the condition \eqref{delta_L_H} for $H$ and its monotonicity, we obtain the following estimate:
\begin{align*}
    &\langle H(\tilde{z}_k^t), z_k^t - z\rangle + \langle H({z}_k^{t-1}), \tilde{z}_k^t - z\rangle\\
    =&\langle  H({z}_k^{t-1})-H(\tilde{z}_k^t), \tilde{z}_k^t- z_k^t\rangle +\langle H( \tilde{z}_k^t),  \tilde{z}_k^t-z\rangle\\
        \geq& -\frac{M}{2}\|z_k^{t-1}- \tilde{z}_k^t\| - \frac{M}{2}\|\tilde{z}_k^t-z_k^t\| - \delta + \langle H(z), \tilde{z}_k^t-z\rangle\\
    \geq& -MV(z_k^{t-1}, \tilde{z}_k^t) - MV(\tilde{z}_k^t, z_k^t) - \delta + \langle H(z), \tilde{z}_k^t-z\rangle, \forall z\in Z.
\end{align*}
Uniting the above two inequalities we have for all $z\in Z$
    \begin{align*}
        &\langle G(\underline{z}_k)+H(z), \tilde{z}_k^t - z \rangle\\
        \leq & \beta_k\left(V(z_{k-1}, z) - V(z_k^t, z) - V(z_{k-1}, \tilde{z}_k^t)-V(\tilde{z}_k^t, z^t_k)\right)\\
        & + \eta_k^t \left(V(z_{k}^{t-1}, z) - V(z_k^t, z) - V(z_{k}^{t-1}, \tilde{z}_k^t)-V(\tilde{z}_k^t, z^t_k)\right)\\
        &+M V(z_k^{t-1}, \tilde{z}_k^t) + M V(\tilde{z}_k^t, z_k^t) + \delta \\
        \leq & \beta_k V(z_{k-1}, z) -\beta_k V(z_{k-1}, \tilde{z}_k^t) + \eta_k^t V(z_{k}^{t-1}, z) - (\beta_k+\eta_k^t) V(z_k^t, z) + \delta.
    \end{align*}
Averaging relations above for $t=1,\dots, T_k$ and using convexity of $V(z_{k-1}, \cdot)$ and the definition $\tilde{z}_k = \frac{1}{T_k}\sum\limits_{t=1}^{T_k} \tilde{z}_k^t$, we obtain the following estimate:
    \begin{align*}
        &\langle G(\underline{z}_k)+H(z), \tilde{z}_k^t - z \rangle\\
        \leq & \beta_k\left(V(z_{k-1}, z) - V(z_k^t, z) - V(z_{k-1}, \tilde{z}_k^t)-V(\tilde{z}_k^t, z^t_k)\right)\\
        & + \eta_k^t \left(V(z_{k}^{t-1}, z) - V(z_k^t, z) - V(z_{k}^{t-1}, \tilde{z}_k^t)-V(\tilde{z}_k^t, z^t_k)\right)\\
        &+M V(z_k^{t-1}, \tilde{z}_k^t) + M V(\tilde{z}_k^t, z_k^t) + \delta \\
        \leq & \beta_k V(z_{k-1}, z) -\beta_k V(z_{k-1}, \tilde{z}_k) \\
        &+\frac{\eta_k^1}{T_k}V(z_{k}^{0}, z)+ \sum\limits_{t=2}^{T_k} (\eta_k^t-\beta_k -\eta_k^{t-1}) V(z_{k}^{t-1}, z) - \frac{\beta_k+\eta_k^{T_k}}{T_k} V(z_k^{T_k}, z) + \delta.
    \end{align*}
Using the relation between $\beta_k$ and $\eta_k^t$ and recalling that $z_k^0 = z_{k-1}$ and $z_k^{T_k}=z_k$ we obtain the result of this lemma immediately.
\end{proof}

The significant difference with Lemma 2.2 from the work \cite{Lan_MPS} is the last term $\delta$. Further, we will consider its influence on the convergence of mirror-prox sliding. Now, we can obtain the convergence of method \ref{alg:mps_deterministic} using the following quantities:
$$\Gamma_k = \begin{cases}1&k=1,\\(1-\gamma_k)\Gamma_{k-1}&k>1.\end{cases}$$

\begin{lemma}
\label{lemma:mps_deterministic}
    Let the conditions of Lemma \ref{lemma:int_determenistic}  and the following conditions hold:
    $$\gamma_1=1, \gamma_k\in [0,1], \beta_k\geq L\gamma_k, \forall k\geq 1,$$
    $$\frac{\gamma_k}{\Gamma_k}\left(\beta_k +\frac{\eta_k^1}{T_k}\right)\leq \frac{\gamma_{k-1}(\beta_{k-1}+\eta_{k-1}^{T_k-1})}{\Gamma_{k-1}T_{k-1}},\forall k\geq 2.$$
    We have
    $$Q(\overline{z}_N, z)\leq \Gamma_N\left(\beta_1+\frac{\eta^1_1}{T_1}\right)V(z_0,z)-\frac{\gamma_N(\beta_N+\eta^{T_N}{N})}{T_N}V(z_N,z)+\delta \Gamma_N \sum\limits_{k=1}^N \frac{\gamma_k}{\Gamma_k},\forall z\in Z.$$
\end{lemma}
\begin{proof}
Combining the results of Lemmas \ref{lemma:ext_determenistic} and \ref{lemma:int_determenistic} and recalling the strong-convexity of $V(\cdot, \cdot)$ we obtain the following inequality:
    \begin{align*}
    &\left[G(\overline{z}_k)-G(z)+\langle H(z), \overline{z}_k-z\rangle\right]\\
    & -(1-\gamma_k)\left[G(\overline{z}_{k-1})-G(z)+\langle H(z), \overline{z}_{k-1}-z\rangle\right]\\
    \leq & \gamma_k\left[\left(\beta_k+\frac{\eta_k^1}{T_k}\right)V(z_{k-1}, z)- \frac{\beta_k+\eta_k^{T_k}}{T_k} V(z_k, z) + \delta\right], \forall z\in Z.
\end{align*}
Dividing the above relation by $\Gamma_k$, summing from $k = 1 \dots, N$, and recalling our assumption
of parameters, we conclude
$$Q(\overline{z}_N, z)\leq \Gamma_N\left(\beta_1+\frac{\eta^1_1}{T_1}\right)V(z_0,z)-\frac{\gamma_N(\beta_N+\eta^{T_N}{N})}{T_N}V(z_N,z)+\delta \Gamma_N \sum\limits_{k=1}^N \frac{\gamma_k}{\Gamma_k}.$$
\end{proof}

We see here that the error accumulates as $\delta \Gamma_N \sum\limits_{k=1}^N \frac{\gamma_k}{\Gamma_k}.$ Further, we demonstrate that the parameters proposed in \cite{Lan_MPS} allows value $\Gamma_N \sum\limits_{k=1}^N \frac{\gamma_k}{\Gamma_k}$ to be constant. Moreover, we obtain the following result:
\begin{theorem}
\label{theorem_mps_deterministic}
    Suppose that parameters in the outer iterations of Algorithm \ref{alg:mps_deterministic} are set to:
    $$\gamma_k=\frac{2}{k+1},\beta_k=\frac{2L}{k},T_k=\left\lceil\frac{kM}{L}\right\rceil, \eta_k^t=\beta_k(t-1)+\frac{LT_k}{k}.$$
    In order to compute an approximate solution $\overline{z}_N$ such that $\sup_{z\in Z}Q(\overline{z}_N,z)\leq \varepsilon+\delta$, the number of evaluations of
gradients $\nabla G(\cdot)$ and operators $H(\cdot)$ are bounded by
$$N_{\nabla G}=O\left(\sqrt{\frac{L\Omega_{z_0}^2}{\varepsilon}}\right),\quad N_H=O\left(\frac{M\Omega_{z_0}^2}{\varepsilon}+\sqrt{\frac{L\Omega_{z_0}^2}{\varepsilon}}\right),$$
respectively, where $\Omega_{z_0}^2=\sup_{z\in Z} V(z_0, z).$
\end{theorem}
\begin{proof}
    In the work \cite{Lan_MPS} these parameters was proven to satisfy conditions of Lemma \ref{lemma:mps_deterministic}.     For given parameters we have that the first two terms in estimation from Lemma \ref{lemma:mps_deterministic}
 can be estimated in the following way:
$$\Gamma_N\left(\beta_1+\frac{\eta^1_1}{T_1}\right)V(z_0,z)-\frac{\gamma_N(\beta_N+\eta^{T_N}{N})}{T_N}V(z_N,z)\leq \frac{6L\Omega_{z_0}^2}{N^2}$$
    Also, note that $\Gamma_k= \frac{2}{k(k+1)}.$ Therefore, $\sum\limits_{k=1}^N \frac{\gamma_k}{\Gamma_k} = \sum\limits_{k=1}^N k = \frac{N(N+1)}{2}=\frac{1}{\Gamma_N}$ and
    $$\delta \Gamma_N \sum\limits_{k=1}^N \frac{\gamma_k}{\Gamma_k}=\delta, \forall N\geq 1.$$
    So, we obtain the following estimation for $\sup_{z\in Z}Q(\overline{z}_N,z)$:
    $$\sup_{z\in Z}Q(\overline{z}_N,z)\leq \frac{6L\Omega_{z_0}^2}{N^2}+\delta.$$
    From this we obtain the results of Theorem \ref{theorem_mps_deterministic}.
\end{proof}

So, we can see that mirror-prox sliding form work \cite{Lan_MPS} can be applied for non-lipschitz operator. At the same time, the accuracy of solution can not be better than $\delta$.

\vspace{0.5cm}
\noindent\textbf{{Stochastic case}}. 
\begin{algorithm}[H]
	\caption{The stochastic mirror-prox sliding (SMPS) method for solving VI \eqref{vi_problem}.} \label{alg:mps_stochastic}
	\begin{algorithmic}[1]
		\REQUIRE Initial point $z_0 \in Z$, external steps number $N$,  sequence on internal steps number $\{T^k\}_{k=1}^n$, sequences of parameters $\{\beta_k\}_{k=1}^n, \{\eta_k^t\}_{t\leq T_k, k\leq N}, \{\gamma_k\}_{k=1}^N$
		
		\STATE Modify step 5 of Algorithm \ref{alg:mps_deterministic} in the following way:
		$$\tilde{z}_k^t = \argmin\limits_{z\in Z} \langle \nabla G(\underline{z}_k)+\mathcal{H}(z_k^{t-1};\zeta_k^{2t-1}), z\rangle + \beta_k V(z_{k-1}, z) + \eta_k^t V(z_k^{t-1}, z),$$        
		$${z}_k^t = \argmin\limits_{z\in Z} \langle \nabla G(\underline{z}_k)+\mathcal{H}(\tilde{z}_k^t; \zeta_k^{2t}), z\rangle + \beta_k V(z_{k-1}, z) + \eta_k^t V(z_k^{t-1}, z).$$
	\end{algorithmic}
\end{algorithm}
Further, let us consider stochastic case. For this we use modification of MPS method from \cite{Lan_MPS} (see Algorithm \ref{alg:mps_stochastic}). This algorithm use operator $\mathcal{H}\left(z; \zeta\right)$. It depends on the deterministic variable $z$ and random variable $\zeta.$ Operator $\mathcal{H}$ is random estimation of the operator $H$ from the problem \eqref{vi_problem}. We prove that the stochastic version of Algorithm \ref{alg:mps_deterministic} works for non-lipschitz operator too.

\begin{assumption}
\label{assump:h}
    Suppose that $\zeta_k^s$'s are independently random samples and the stochastic operator $\mathcal{H}$ satisfies unbiasedness $\mathbb{E}_{\zeta^{k}_{2t-1}} \left[\mathcal{H}\left(z^{k}_{t-1}; \zeta^{k}_{2t-1}\right)\right]=H(z^k_{t-1})$ and $\mathbb{E}_{\zeta^{k}_{2t}} \left[\mathcal{H}\left(\tilde{z}_k^t; \zeta^{k}_{2t}\right)\right]=H(\tilde{z}_k^t)$ and bounded variance $\mathbb{E}_{\zeta^{k}_{2t-1}} \left\|\mathcal{H}\left(z^{k}_{t-1}; \zeta^{k}_{2t-1}\right)-H(z^k_{t-1})\right\|^2\leq \sigma^2$ and $\mathbb{E}_{\zeta^{k}_{2t}} \left\|\mathcal{H}\left(\tilde{z}_k^t; \zeta^{k}_{2t}\right)-H(\tilde{z}_k^t)\right\|^2\leq \sigma^2$ for all $t=\overline{1, T_k}, k=\overline{1, N}$.
\end{assumption}

Further, to analyze Algorithm \ref{alg:mps_stochastic} we introduce the following estimates:
\begin{align}
    \label{delta_def}
    \Delta_k^{2t-1}&:=\mathcal{H}(z_k^{t-1}; \zeta_k^{2t-1})-H(z_k^{t-1})~~\text{ and }~~\Delta_k^{2t}:=\mathcal{H}(\tilde{z}_k^{t}; \zeta_k^{2t})-H(\tilde{z}_k^{t-1}), \\
    w_k^t &= \argmin_{z\in Z} \left( -\langle \Delta_k^{2t}, z\rangle +\beta_k V(w_{k-1}, z) + \eta_k^t V(w_k^{t-1}, z) \right), \\
    \label{w_tilde_def}
    \tilde{w}_k^t &= \frac{\eta_k^t}{\beta_k+\eta_k^t} w_k^{t-1} + \frac{\beta_k}{\beta_k+\eta_k^t} w_{k-1}.
\end{align}

Note, that external iterations are still described by Lemma \ref{lemma:ext_determenistic}. Using the above notation, we can obtain the following lemma for internal iterations of Algorithm \ref{alg:mps_stochastic}.

\begin{lemma}
    \label{stoch_lemma}
    If the parameters of Algorithm \ref{alg:mps_stochastic} satisfy $M\leq \beta_k + \eta_k^t, \forall k\geq 1, t\geq 1$ and $\eta_k^t\leq \beta_k + \eta_k^{t-1}, \forall k\geq 1, t\geq 2$ then
    \begin{align*}
        &\langle G(\underline{z}_k) + H(z), \tilde{z}_k-z\rangle\\
        \leq & -\beta_k V(z_{k-1}, \tilde{z}_k) + \left(\beta_k + \frac{\eta_k^1}{T_k}\right)\left(V(z_{k-1}, z) + V(w_{k-1}, z)\right) \\
        &- \frac{\beta_k + \eta_k^{T_k}}{T_k}\left(V(z_k, z) + V(w_k, z)\right) + \frac{1}{T_k} \sum\limits_{t=1}^{T_k} \frac{1}{2M}\left(2\|\Delta_k^{2t-1}\|^2 + 3\|\Delta_k^{2t}\|^2 \right)\\ 
        &- \frac{1}{T_k} \sum\limits_{t=1}^{T_k}\langle \Delta_k^{2t}, \tilde{z}_k^t - \tilde{w}_k^t\rangle+\delta, \forall z\in Z,
    \end{align*}
where $\tilde{w}_k^t$ is defined in \eqref{w_tilde_def}.
\end{lemma}

\begin{proof}

    In the proof of Lemma 3.1 in \cite{Lan_MPS} the authors proved three following inequalities. Firstly, it is the estimate:
    \begin{equation}
    \label{eq_3_1_1}
        \begin{aligned}
            &\langle G(\underline{z}_k), \tilde{z}_k^t - z \rangle + \langle\mathcal{H}(\tilde{z}_k^t;\zeta_k^{2t}), z_k^t- z\rangle \\ &+ \langle\mathcal{H}({z}_k^{t-1};\zeta_k^{2t-1}), \tilde{z}_k^t-z_k^t\rangle\\
            \leq & \beta_k V(z_{k-1}, {z})- \beta_k V(z_{k-1}, \tilde{z}_k^t) + \eta_k V(z_{k}^{t-1}, {z}) - (\beta_k+\eta_k^t) V(z_{k}^t, {z})\\
            &-(\beta_k+\eta_k^t)V(\tilde{z}_k^t, z_k^t) - \eta_k^t V(z_{k}^{t-1}, \tilde{z}_k^t), \forall z\in Z,
        \end{aligned}
    \end{equation}
    Secondly, it is the estimate for $\langle \Delta_k^{2t}, w_k^t - z \rangle$:
    \begin{equation}
    \label{eq_3_1_2}
        \begin{aligned}
            &-\langle \Delta_k^{2t}, w_k^t - z \rangle\\
            \leq & \beta_k\left(V(w_{k-1}, z) - V(w_{k-1}, w_k^t)-V(w_k^t, z)\right)\\
            &+ \eta_k\left(V(w_k^{t-1}, z)-V(w_k^{t-1},w_k^t)-V(w_k^t,z)\right), \forall z\in Z,
        \end{aligned}
    \end{equation}
    Finally, we have the following estimate for $\langle \Delta_k^{2t}, \tilde{w}_k^t-w_k^t \rangle$:
    \begin{equation}
    \label{eq_3_1_3}
        \begin{aligned}
            &-\langle \Delta_k^{2t}, \tilde{w}_k^t-w_k^t \rangle \\
            \leq&\frac{1}{2(\beta_k+\eta_k^t)}\|\Delta_k^{2t}\|^2 + \beta_k V(w_{k-1}, w_k^t) + \eta_k^t V(w_k^{t-1}, w_k^t), \forall z\in Z.
        \end{aligned}
    \end{equation}
    
Further, using definition of $\Delta_k^{2t}$ and $\Delta_k^{2t-1}$ $\eqref{delta_def}$ and condition \eqref{delta_L_H} and monotonicity for operator $H(\cdot)$ we obtain the following estimation:
\begin{align*}
     &\langle\mathcal{H}(\tilde{z}_k^t;\zeta_k^{2t}), z_k^t- z\rangle + \langle\mathcal{H}({z}_k^{t-1};\zeta_k^{2t-1}), \tilde{z}_k^t-z_k^t\rangle\\
     =& \langle\mathcal{H}({z}_k^{t-1};\zeta_k^{2t-1}) - \mathcal{H}(\tilde{z}_k^{t};\zeta_k^{2t}), \tilde{z}_k^t-z_k^t\rangle + \langle H(\tilde{z}_k^t), \tilde{z}_k^t-z\rangle + \langle \Delta_k^{2t}, \tilde{z}_k^t-z\rangle \\
     =& \langle H({z}_k^{t-1}) - H(\tilde{z}_k^{t}), \tilde{z}_k^t-z_k^t\rangle + \langle \Delta_k^{2t-1} - \Delta_k^{2t}, \tilde{z}_k^t-z_k^t\rangle + \langle H(z), \tilde{z}_k^t-z\rangle + \langle \Delta_k^{2t}, \tilde{z}_k^t-z\rangle\\
     \geq &-\frac{M}{2}\|\tilde{z}_k^t- z_k^{t-1}\|^2-\frac{M}{2}\|\tilde{z}_k^t- z_k^t\|^2 - \delta + \langle \Delta_k^{2t-1} - \Delta_k^{2t}, \tilde{z}_k^t-z_k^t\rangle\\ & + \langle H(z), \tilde{z}_k^t-z\rangle + \langle \Delta_k^{2t}, \tilde{z}_k^t-z\rangle \\
     \geq &-MV(z_k^{t-1}, \tilde{z}_k^t)-MV(\tilde{z}_k^t, z_k^t) - \delta - \frac{\|\Delta_k^{2t-1}+\Delta_k^{2t}\|^2}{2M} - \frac{M}{2}\|\tilde{z}_k^t - z_k^t\|^2\\ & + \langle H(z), \tilde{z}_k^t-z\rangle + \langle \Delta_k^{2t}, \tilde{z}_k^t-z\rangle\\
     \geq &-MV(z_k^{t-1}, \tilde{z}_k^t)-2MV(\tilde{z}_k^t, z_k^t) - \delta - \frac{1}{M}(\|\Delta_k^{2t-1}\|^2+\|\Delta_k^{2t}\|^2)\\ & + \langle H(z), \tilde{z}_k^t-z\rangle + \langle \Delta_k^{2t}, \tilde{z}_k^t-z\rangle.\end{align*}
Uniting the inequality above with \eqref{eq_3_1_1} and the inequality between $M, \beta_k, \eta_k^t$ we obtain the following estimate:
    \begin{equation}
    \label{eq_3_1_1_new}
        \begin{aligned}
            &\langle G(\underline{z}_k)+ H(z), \tilde{z}_k^t - z \rangle \\
            \leq & \beta_k V(z_{k-1}, {z})- \beta_k V(z_{k-1}, \tilde{z}_k^t) + \eta_k V(z_{k}^{t-1}, {z}) - (\beta_k+\eta_k^t) V(z_{k}^t, {z})\\
            & +\frac{1}{M}(\|\Delta_k^{2t-1}\|^2+\|\Delta_k^{2t}\|^2)-\langle\Delta_k^{2t}, \tilde{z}_k^t-z\rangle+\delta, \forall z\in Z.
        \end{aligned}
    \end{equation}

Uniting the relation above with \eqref{eq_3_1_2} and \eqref{eq_3_1_3} we obtain the following estimate:
    \begin{align*}
        &\langle G(\underline{z}_k) + H(z), \tilde{z}_k-z\rangle\\
        \leq & \beta_k\left(V(z_{k-1},z)+V(w_k^{t-1}, z)\right)-\beta_k V(z_{k-1}, \tilde{z}_k^t)\\
        &+\eta_k^t (V(z_{k}^{t-1}, z)+V(w_k^{t-1}, z))-(\beta_k+\eta_k^t) \left(V(z_k^t, z)+V(w_k^t, z)\right)\\
        &+\frac{1}{2M}(2\|\Delta_k^{2t-1}\|^2+3\|\Delta_k^{2t}\|^2)-\langle\Delta_k^{2t}, \tilde{z}_k^t-\tilde{w}_k^t \rangle+\delta, \forall z\in Z.
    \end{align*}
Taking the average of estimation above from $t=\overline{1, T_k}$ and recalling convexity of Bregman divergence  $V(z, \cdot), \forall z$ we obtain the statement of this lemma.
\end{proof}

We see, that dependence on $\delta$ is similar to deterministic case. Further, we can obtain result about convergence of Algorithm \ref{alg:mps_stochastic} and prove the same accumulation of $\delta$.

\begin{lemma}
\label{lemma:mps_stochastic}
    Let Assumption \ref{assump:h} hold. Let the conditions of Lemma \ref{stoch_lemma}  and the following conditions hold:
    $$\gamma_1=1, \gamma_k\in [0,1], \beta_k\geq L\gamma_k, \forall k\geq 1,$$
    $$\frac{\gamma_k}{\Gamma_k}\left(\beta_k +\frac{\eta_k^1}{T_k}\right)\leq \frac{\gamma_{k-1}(\beta_{k-1}+\eta_{k-1}^{T_k-1})}{\Gamma_{k-1}T_{k-1}},\forall k\geq 2.$$
    We have
    $$\mathbb{E}\left[\sup_{z\in Z}Q(\overline{z}_N, z)\right]\leq \Gamma_N\left(\beta_1+\frac{\eta^1_1}{T_1}\right)\Omega_{z_0}^2 + \Gamma_N \frac{5\sigma^2}{2M} \sum\limits_{k=1}^n \frac{\gamma_k}{\Gamma_k} +\delta \Gamma_N \sum\limits_{k=1}^N \frac{\gamma_k}{\Gamma_k}.$$
\end{lemma}
\begin{proof}
Combining the results of Lemmas \ref{lemma:ext_determenistic} and \ref{stoch_lemma} and recalling the strong-convexity of $V(\cdot, \cdot)$ we obtain the following inequality:
    \begin{align*}
    &\left[G(\overline{z}_k)-G(z)+\langle H(z), \overline{z}_k-z\rangle\right]\\
    & -(1-\gamma_k)\left[G(\overline{z}_{k-1})-G(z)+\langle H(z), \overline{z}_{k-1}-z\rangle\right]\\
    \leq & \gamma_k\left[\left(\beta_k + \frac{\eta_k^1}{T_k}\right)\left(V(z_{k-1}, z) + V(w_{k-1}, z)\right) - \frac{\beta_k + \eta_k^{T_k}}{T_k}\left(V(z_k, z) + V(w_k, z)\right)\right]\\
    & + \frac{\gamma_k}{2MT_k} \sum\limits_{t=1}^{T_k} \left(2\|\Delta_k^{2t-1}\|^2 + 3\|\Delta_k^{2t}\|^2 \right)- \frac{\gamma_k}{T_k} \sum\limits_{t=1}^{T_k}\langle \Delta_k^{2t}, \tilde{z}_k^t - \tilde{w}_k^t\rangle+\gamma_k \delta, \forall z\in Z.
\end{align*}
Dividing the above relation by $\Gamma_k$, summing from $k = 1 \dots, N$, and recalling our assumption
of parameters, we immediately obtain the result of the current lemma.
\end{proof}

As we proved for the deterministic case for $\Gamma_k =\frac{2}{k(k+1)}$ and $\gamma_k = \frac{2}{k+1}$ the sum $\Gamma_N \sum\limits_{k=1}^N \frac{\gamma_k}{\Gamma_k}$ equals 1. So, we have the following result.

\begin{theorem}
\label{theorem_mps_stochastic}
    Suppose that Assumption \ref{assump:h} holds and parameters in the outer iterations of Algorithm \ref{alg:mps_stochastic} are set to:
    $$\gamma_k=\frac{2}{k+1},\beta_k=\frac{2L}{k},T_k=\left\lceil\frac{\sqrt{3}kM}{L}+\frac{N k^2\sigma^2}{\Omega_{z_0}^2L^2}\right\rceil, \eta_k^t=\beta_k(t-1)+\frac{LT_k}{k}.$$
    In order to compute an approximate solution $\overline{z}_N$ such that $\mathbb{E}\sup_{z\in Z}Q(\overline{z}_N,z)\leq \varepsilon+\delta$, the number of evaluations of
gradients $\nabla G(\cdot)$ and operators $H(\cdot)$ are bounded by
$$N_{\nabla G}=O\left(\sqrt{\frac{L\Omega_{z_0}^2}{\varepsilon}}\right),\quad N_H=O\left(\frac{M\Omega_{z_0}^2}{\varepsilon}+\frac{\sigma^2\Omega_{z_0}^2}{\varepsilon^2}+\sqrt{\frac{L\Omega_{z_0}^2}{\varepsilon}}\right),$$
respectively, where $\Omega_{z_0}^2=\sup_{z\in Z} V(z_0, z).$
\end{theorem}
\begin{proof}
The proof of this statement is similar to proof of Corollary 3.3 from the work \cite{Lan_MPS}. The main difference is parameter $\delta$ and proof that it does not accumulate.

Note, that the parameters in this theorem satisfies condition from Lemma \ref{lemma:mps_stochastic}. Therefore, we have the following inequalities:
    \begin{align*}
        \mathbb{E}\left[\sup_{z\in Z}Q(\overline{z}_N, z)\right]&\leq \Gamma_N\left(\beta_1+\frac{\eta^1_1}{T_1}\right)\Omega_{z_0}^2 + \Gamma_N \frac{5\sigma^2}{2M} \sum\limits_{k=1}^n \frac{\gamma_k}{\Gamma_k} +\delta \Gamma_N \sum\limits_{k=1}^N \frac{\gamma_k}{\Gamma_k}\\
        &= \leq \Gamma_N\left(\beta_1+\frac{\eta^1_1}{T_1}\right)\Omega_{z_0}^2 + \Gamma_N \frac{5\sigma^2}{2M} \sum\limits_{k=1}^n \frac{\gamma_k}{\Gamma_k} +\delta\\
        &\leq \frac{12L\Omega_{z_0}}{N(N+1)} + \frac{1}{N(N+1)}\sum\limits_{k=1}^N \frac{7\sigma^2 k^2}{LT_k}+\delta\\
        &\leq\frac{19L\Omega_{z_0}}{N^2} +\delta.
    \end{align*}
    Therefore, the given in this theorem $N_{\nabla G}$ calculations of $\nabla G$ is sufficient to approach required accuracy $\varepsilon+\delta.$ At the same time, the number of $\mathcal{H}$ calculations can be found in the following way:
    \begin{align*}
        N_H &= 2\sum\limits_{k=1}^{N_{\nabla G}} T_k\\
        &\leq 2\sum\limits_{k=1}^{N_{\nabla G}}\left(1+\frac{\sqrt{3}kM}{L}+\frac{N k^2\sigma^2}{\Omega_{z_0}^2L^2}\right)\\
        &=O\left(\frac{M\Omega_{z_0}^2}{\varepsilon}+\frac{\sigma^2\Omega_{z_0}^2}{\varepsilon^2}+\sqrt{\frac{L\Omega_{z_0}^2}{\varepsilon}}\right).
    \end{align*}
\end{proof}

So, we obtained that deterministic and stochastic Mirror-prox sliding converge to accuracy $\varepsilon+\delta$ for non-lipschitz operator $H$ with the same speed as for the lipschitz operator to accuracy $\varepsilon.$ In other words, this method can be applied for non-lipschitz operator too, but accuracy is limited by $\delta\geq 0$ from the inequality \eqref{delta_L_H}.

\section{Mirror-prox sliding for decentralized optimization}

\subsection{Decentralized problem Statement}\label{sec:decentr_spp}

Recall a sum-type saddle-point problem (SPP).
\begin{equation*}
    \begin{aligned}
	   \min_{x \in X} \max_{y \in Y} \left[f(x, y)=\frac{1}{m} \sum_{i=1}^m f_i(x, y)\right], 
	\end{aligned}
\end{equation*}
Introduce $\bx = \col(x_1\ldots x_m) \in \mathbb{R}^{md_x},~ \by = \col(y_1, \ldots, y_m)  \in \mathbb{R}^{md_y}$. We will note $X_m = \{\bx = \col(x_1, \ldots, x_m):~ x_1\in X,~\ldots,~ x_m\in\ X\}$ and analogously $Y_m = \{\by = \col(y_1, \ldots, y_m):~ y_1\in Y,~\ldots,~ y_m\in\ Y\}$. Also define $F(\bx, \by) = \sum_{i=1}^m f_i(x_i, y_i)$. We consider a decentralized saddle-point problem (SPP) with consensus constraints.
\begin{equation}\label{eq:problem_consensus_constraints}
    \begin{aligned}
	            \min_{\bx\in X_m}\max_{\by\in Y_m}~ &F(\bx, \by), \\
	    \text{s.t. } &x_1 = \ldots = x_m,  \\
	    & y_1 = \ldots = y_m.
	    \end{aligned}
\end{equation}
It is convenient to rewrite the consensus constraints as affine constraints using gossip matrices.
{
\begin{assumption}\label{assum:gossip_matrix}
	Gossip matrix $W$ satisfies the following properties.
	\item 1. $[W]_{ij} = 0$ if $i\neq j$ and $(i, j)\notin\cE$.
	\item 2. $W = W^\top$.
	\item 3. $Wx = 0$ if and only if $x_1 = \ldots = x_m$.
\end{assumption}
A typical choice for the gossip matrix is graph Laplacian $\cL(\cG)$, which is a matrix with entries
\begin{align*}
	[\cL]_{ij} = 
	\begin{cases}
			\deg(i), &\text{if } i = j, \\
			-1, &\text{if } (i, j)\in\cE, \\
			0, &\text{else}.
		\end{cases}
\end{align*}
Gossip matrix enables to rewrite the consensus constraints due to its kernel property. Let us introduce gossip matrices $\widetilde W_x$ and $\widetilde W_y$ for variables $x$ and $y$, respectively. Also define $\widetilde\mW_x = \widetilde W_x\otimes \mI,~ \widetilde\mW_y = \widetilde W_y\otimes\mI$. Then consensus constraints can be rewritten as $\mW_x\bx=\mathbf{0},~ \mW_y\by=\mathbf{0}$. Then we denote $\mW_x:=\widetilde{\mW}_x^{1/2}$ and $\mW_y:=\widetilde{\mW}_y^{1/2}$.} After that, we replace the consensus constraints of \eqref{eq:problem_consensus_constraints} with the affine constraints $\mW_x \bx=\mathbf{0}$, $\mW_y \by=\mathbf{0}$. Finally, we rewrite the affine constraints as a penalty and add it to the functional (we describe it in Section \ref{sec:spp_penalty}):
\begin{align}\label{eq:spp_penalty}
    \min_{\bx \in X_m} \max_{\by \in Y_m}~ \left[F_R(\bx, \by):=F(\bx, \by) + \frac{R_\alpha^2}{\eps}\norm{\mW_x \bx}_2^2 - \frac{R_\beta^2}{\eps}\norm{\mW_y \by}_2^2\right].
\end{align}
Here $\eps$ denotes the desired quality of the solution. The coefficients $R_\alpha, R_\beta$ should be chosen large enough so that solution of penalized problem \eqref{eq:spp_penalty} is a quite good approximation of non-penalized problem \eqref{eq:problem_consensus_constraints}. We discuss the choice of regularization coefficients $R_\alpha$ and $R_\beta$ in Section \ref{sec:spp_penalty}.

\subsection{Moving the affine constraints of SPP to penalty}\label{sec:spp_penalty}

In this section, we describe how regularization coefficients $R_\alpha, R_\beta$ in problem \eqref{eq:spp_penalty} affect the distance between the solutions of penalized problem \eqref{eq:spp_penalty} and original constrained problem \eqref{eq:problem_consensus_constraints}. We formulate the result for a general convex-concave SPP with affine constraints.


First, we formulate a result for (not min-max) optimization. Consider some optimization problem over network with gossip matrix $\mW$:
\begin{align}\label{eq:linear_constraints}
   \min_{\substack{\mW \bx=\mathbf{0}\\\bx\in Q}} u(\bx).
\end{align}
This problem is independent of the initial saddle-point problem. For this problem, we can to unit results  from works \cite{gorbunov2019} and \cite{lan2017communication} to the following lemma.
\begin{lemma}\label{lem:penalty}
    Let $u(\bx)$ be a convex function defined over a convex non-empty set $Q$ and let $\kernel \mW\ne \{0\}$ in problem \eqref{eq:linear_constraints} and $R$ be a bound on the dual solution size, i.e. there exists such solution $q^*$ of the problem dual to \eqref{eq:linear_constraints}, such that $\normtwo{q^*}\le R$. Besides, introduce the function $U(\bx) := u(\bx) + \frac{R^2}{\eps}\|\mW \bx\|_2^2.$
 and two points $\ds \bx_0 = \argmin_{\bx\in Q} U(\bx)$, $\ds \bx^* = \argmin_{\substack{\mW \bx=\mathbf{0} \\ \bx\in Q}} u(\bx)$. Then the following statements holds:
	\begin{enumerate}
		\item If $\displaystyle U(\tilde \bx) - \min_{\bx\in Q} U(\bx)\le \eps$, then we have $u(\tilde \bx) -  u(\bx^*) \le \eps$ and $\normtwo{\mW \tilde \bx}\le 2\eps / R.$
		\item Let $\partial u(\mathbf{x})$ be a subdifferential of the function $u$ at point $\mathbf{x}$. If $R$ satisfies the following inequality:
			$R^2\ge \dfrac{\sup_{v\in\partial u(\bx^*)}\normtwo{v}^2}{\lambda_{\min}^+(\mW^\top \mW)},$
  then we have such solution $q^*$ of the problem dual to \eqref{eq:linear_constraints}, such that $\|q^*\|\leq R.$  
		\item For points $\bx^*$ and $\bx_0$ the inequality $0 \le u(\bx^*) - U(\bx_0) \le 2\eps$ holds.
	\end{enumerate}
\end{lemma}

\begin{proof}

    By definition of $U$, we have the following inequalities for all $\bx$:
    \begin{align*}
        U(\bx)- U(\bx_0) &= u(\bx) + \frac{R^2}{\varepsilon}\|\mW\bx\|_2^2 - \min_{\bx\in Q}\left[u(\bx) + \frac{R^2}{\varepsilon}\|\mW \bx\|_2^2\right]\\
        &\geq u(\bx) + \frac{R^2}{\varepsilon}\|\mW\bx\|_2^2 - \min_{\mW \bx=\mathbf{0}, \bx\in Q}\left[u(\bx) + \frac{R^2}{\varepsilon}\|\mW\bx\|_2^2\right]\\
        &=u(\bx)-u(\bx^*)+\frac{R^2}{\varepsilon}\|\mW\bx\|_2^2.
    \end{align*}
    Therefore, we have the following estimation:
    \begin{equation}\label{lemma_1_eq}
         U(\bx)- U(\bx_0) \geq u(\bx)-u(\bx^*)+\frac{R^2}{\varepsilon}\|\mW\bx\|_2^2
    \end{equation}
    If $U(\tilde{\bx})- U(\bx_0)\leq \varepsilon$ then $=u(\tilde{\bx})-u(\bx^*)\leq \varepsilon$ too.

   Following the arguments in \cite{lan2017communication}, we denote $\by^*$ the solution of the dual and write the definition of the saddle point.
    \begin{align*}
        u(\bx^*) + \angles{\by, \mW\bx^*} \leq u(\bx^*) + \angles{\by^*, \mW\bx^*} \leq u(\bx) + \angles{\by^*, \mW\bx}.
    \end{align*}
    Since $\mW\bx^* = 0$, we have the following estimation:
    \begin{equation}\label{lemma_2_eq}
        u(\bx) - u(\bx^*)\geq \angles{\by^*, \mW(\bx^* - \bx)} = \angles{-\mW^\top\by^*, \bx - \bx^*}.
    \end{equation}
    From \eqref{lemma_2_eq}  and $\|\by^*\|\leq R$, we have that $u(\bx) - u(\bx^*) \geq -R\|\mW\bx\|.$ 
    Then using this result and \eqref{lemma_1_eq} we have $U(\bx)- U(\bx_0) \geq \frac{R^2}{\varepsilon}\|\mW\bx\|_2^2-R\|\mW\bx\|, \forall \bx.$ In particular, for $\tilde{\bx}$ the estimation $\frac{R^2}{\varepsilon}\|\mW\bx\|_2^2-R\|\mW\bx\| \leq \varepsilon$ holds. Because of it, we conclude that the point $\tilde{\bx}$ satisfies the condition $\|\mW\bx\|\leq \frac{2\varepsilon}{R}$.

    Further, we prove statement 2. Because of \eqref{lemma_2_eq}, $-\mW^\top\by^*\in\partial u(\bx^*)$. If we take the dual solution $\by^*\in (\kernel\mW^\top)^\bot$, we have:
    \begin{align*}
        \norm{\by^*}_2^2\leq \frac{\sup_{v\in\partial u(\bx^*)}\norm{v}_2^2}{\lambda_{\min}^+(\mW^\top\mW)}.
    \end{align*}
    So, it proves the second statement.
    
    Let us prove statement 3. The first inequality is obvious: $U(\bx_0)\le U(\bx^*) = u(\bx^*)$. In order to prove the second inequality, we use the convexity of $u$:
	\begin{align*}
		u(\bx^*) - u(\bx_0)
		&\le \angles{\nabla u(\bx^*), \bx^* - \bx_0}\\
		&\le \normtwo{\nabla u(\bx^*)}\cdot\normtwo{\bx^* - \bx_0}
		\\&\le \normtwo{\nabla u(\bx^*)}\cdot\frac{\normtwo{\mW \bx_0}}{\sqrt{\lambda_{\min}^+(\mW^\top \mW)}}.
	\end{align*}
	Since $\ds U(\bx_0) - \min_{\bx\in X} U(\bx) = 0\le \eps$, from statement 1 it follows that $\normtwo{\mW \bx_0}\le 2\eps/R$. Substituting the bound on $R$ from statement 2, we obtain
	\begin{align*}
		u(\bx^*) - u(\bx_0) &\le \frac{\normtwo{\nabla u(\bx^*)}}{\sqrt{\lambda_{\min}^+(\mW^\top \mW)}}\cdot \frac{2\eps}{R}
		\le 2\eps, \\
		u(\bx^*) - U(\bx_0) &\le 2\eps - \frac{R^2}{\eps}\normtwo{\mW \bx}^2
		\le 2\eps.
	\end{align*}
\end{proof}

{This lemma gives the connection between value $R$ and required accuracy. Note, that it depends on maximal norm of subgradient at point solution and conditional number of communication network.}

{Further, we choose coefficients $R_\alpha$ and $R_\beta$ to bound the solution of corresponding dual problems.}

{
\begin{lemma}\label{lem:r_alpha_beta}
    Let us define the following:
    $$f(\bx)=\max_{\by\in Y_m} F_R(\bx, \by),~~ \bx^*=\argmin_{\bx\in X_m: \mW\bx=\mathbf{0}} f(\bx),$$ $$\by^*(\bx)=\argmax_{\by \in Y_m: \mW\by=\mathbf{0}} F_R(\bx, \by),~~ g_\bx(\by):=F(\bx, \by).$$ If $R_\alpha$ and $R_\beta$ are defined as:
\begin{equation}
\label{R_alpha_beta_det}
      R_\alpha^2 = \frac{\sup_{\mathbf{v}\in \partial f(\bx^*)}\|\mathbf{v}\|^2}{\lambda_{\min}^+(\mW_x^\top\mW_x)},\quad R_\beta^2 = \sup_{\bx\in X_m}\frac{\sup_{\mathbf{v}\in \partial (-g_\bx(\by^*(\bx))} \|\mathbf{v}\|^2 }{\lambda_{\min}^+(\mW_y^\top\mW_y)},
\end{equation}
then the following statements hold:
\begin{enumerate}
    \item there exists the solution $q_x^*$ of dual problem for problem $\ds\min_{\substack{\bx\in X_m\\ \mW_x \bx=\mathbf{0}}} f(\bx)$ such that $\|q_x^*\|\leq R_{\alpha}$.
    \item there exists the solution $q_y^*(\bx)$ of dual problem for problem $\ds\max_{\substack{\by\in Y_m\\ \mW_y \by=\mathbf{0}}} g_{\bx}(\by)$ such that $\|q_y^*(\bx)\|\leq R_{\beta}$ for any $\bx \in X_m.$
\end{enumerate}
\end{lemma}
}
\begin{proof}
    Firstly, note that problem $\ds\min_{\substack{\bx\in X_m\\ \mW_x \bx=\mathbf{0}}} f(\bx)$ is the particular case of problem \eqref{eq:linear_constraints}. So, according to the second statement of Lemma \ref{lem:penalty} we have that for given $R_{\alpha}$ the first statement of this lemma holds.

    At the same time, for any fixed $\bx\in X_m$ the problem $\ds\max_{\substack{\by\in Y_m\\ \mW_y \by=\mathbf{0}}} g_{\bx}(\by)$ is the particular problem of \eqref{eq:linear_constraints} too. Using the second statement of Lemma \ref{lem:penalty} again, we have $\|q_y^*(\bx)\|\leq \frac{\sup_{\mathbf{v}\in \partial (-g_\bx(\by^*(\bx))} \|\mathbf{v}\|^2 }{\lambda_{\min}^+(\mW_y^\top\mW_y)}$ for any fixed $\bx \in X_m.$ Therefore, the following inequality holds:
    $$\|q_y^*(\bx)\|\leq \sup_{\bx\in X_m}\frac{\sup_{\mathbf{v}\in \partial (-g_\bx(\by^*(\bx))} \|\mathbf{v}\|^2 }{\lambda_{\min}^+(\mW_y^\top\mW_y)}=R_{\beta}.$$
\end{proof}

{
Now, we generalize the results of Lemma~\ref{lem:penalty} to saddle point problems.
\begin{lemma}\label{lem:spp_penalty}
	Let $(\tilde\bx, \tilde\by)$ be a pair such that
	\begin{align*}
		\max_{\by\in Y_m} F_R(\tilde\bx, \by) - \min_{\bx\in X_m} F_R(\bx, \tilde\by)\leq \eps,
	\end{align*}
	where $F_R$ is introduced in \eqref{eq:spp_penalty}. Coefficients $R_\alpha$ and $R_\beta$ are defined according to \eqref{R_alpha_beta_det}. Then it holds
	\begin{align*}
	\ds\max_{\substack{\mW_y\by=\mathbf{0}\\ \by\in Y_m}} F(\tilde\bx, \by) - \min_{\substack{\mW_x\bx=\mathbf{0}\\ \bx\in X_m}} F(\bx, \tilde\by)\leq \eps
	\end{align*}
	and $\norm{\mW_x\tilde\bx}_2\leq 2\eps/R_\alpha,~ \norm{\mW_y\tilde\by}_2\leq 2\eps/R_\beta$.
\end{lemma}
}
\begin{proof}    
{
	Firstly, uniting the inequality $\ds\max_{\by\in Y_m} F_R(\tilde\bx, \by) - \min_{\bx\in X_m} F_R(\bx, \tilde\by)\leq \eps$, definition of $R_\beta$ and Lemmas \ref{lem:penalty}, \ref{lem:r_alpha_beta} we obtain that $\|\mW_y\tilde{\by}\|\leq {2\varepsilon}/{R_\beta}.$
}
{	
	Further, let us denote $g(\bx),\by^*(\tilde\bx)$ and $\bx^*$ as in Lemma\ref{lem:r_alpha_beta}. Using condition of this lemma with $\by:=\by^*(\tilde\bx)$ and $\bx:=\bx^*$, we obtain 
	$g(\tilde\bx)-F_R(\bx^*, \tilde\by)\leq \varepsilon.$ At the same time, $F_R(\bx^*, {\by})\leq g(\bx^*)$ holds for all $\by \in Y_m$ by definition of $g.$ Therefore,
	the point $\tilde\bx$ satisfies to the estimation $g(\tilde\bx)-g(\bx^*)\leq 2\varepsilon$ too. Besides, because of definition $R_\alpha$ we have that $\|\mW_x \tilde\bx\|\leq {2\varepsilon}/{R_\alpha}.$
}
{
	Finally, note according to definition $F_R$ and lemma condition we have the following condition for all $\forall \bx \in X_m, \by \in Y_m,$ such that $x_1=\dots=x_m$ and $y_1=\dots y_m$ (i.e. $\mW_x \bx=\mathbf{0},\mW_y\by=\mathbf{0}$):
	$$F(\tilde\bx, \by)-F(\bx, \tilde\by)+\left(\frac{R_\alpha^2}{\varepsilon}\|\mW_x \tilde\bx \|^2 + \frac{R_\beta^2}{\varepsilon}\|\mW_y \tilde\by \|^2\right)\leq \varepsilon, $$
	Using this inequality, we conclude 
	$F(\tilde\bx, \by)-F(\bx, \tilde\by)\leq \varepsilon$ for all $\forall \bx \in X_m, \by \in Y_m,$ such that $\mW_x \bx=\mathbf{0},\mW_y\by=\mathbf{0}.$
}
\end{proof}

\subsection{Main Result}
\label{sec:main_result}

Given reformulation \eqref{eq:spp_penalty}, we apply the results of \cite{Lan_MPS}. We let $\bz = \col(\bx, \by)\in \mathbb{R}^{m(d_x+d_y)}$. We will also denote $Z_m = \{\bz = \col(\bx, \by):~ \bx\in X,~ \by\in Y\}$. In the denotations of paper \cite{Lan_MPS}, we have:
\begin{align*}
    H(\bz) =
    \begin{pmatrix}
	        \nabla_\bx F(\bx, \by) \\ -\nabla _\by F(\bx, \by)
	    \end{pmatrix},~
    G(\bz) = \frac{R_\alpha^2}{\eps}\norm{\mW_x\bx}_2^2 + \frac{R_\beta^2}{\eps}\norm{\mW_y\by}_2^2,
\end{align*}
Note, that if the function $F$ is smooth function, than $H$ is Lipschitz operator and the result from original work can be applied.


Note, that in new notation the problem \eqref{eq:spp_penalty} is equivalent to the following strong variational inequality (VI):
\begin{equation}\label{strong_vi_problem_bz}
    \langle H (\bz)+\nabla G(\bz), \bz^*-\bz\rangle \leq 0, \forall \bz \in Z_m.
\end{equation}
{
At the same time, the operator $K(\bz)=H (\bz)+\nabla G(\bz)$ is monotone, i.e. 
$$\langle K(\bz_1)-K(\bz_2), \bz_1-\bz_2 \rangle \geq 0,\forall \bz_1, \bz_2 \in Z_m.$$
It means that the problem of strong VI \eqref{strong_vi_problem_bz} is equivalent to the following weak VI:
\begin{equation}\label{vi_problem_bz}
    \langle H (\bz)+\nabla G(\bz), \bz^*-\bz\rangle \leq 0, \forall \bz \in Z_m.
\end{equation}
}

We impose the following assumption on operator $H$.
\begin{assumption}\label{assum:field_h_distr}
There exist constants $\delta, M \geq 0$ such that for any $\bz_1, \bz_2, \bz_3\in Z_m$ it holds
\begin{equation}
\label{delta_L_H_bz}
    \langle H(\bz_1)-H(\bz_2), \bz_1-\bz_3\rangle \leq \frac{M}{2}\| \bz_1- \bz_2\|^2+\frac{M}{2}\|\bz_1-\bz_3\|^2+\delta.
\end{equation}
\end{assumption}
\begin{remark}\label{remark:assumption_examples}
    Let us consider two cases when Assumption~\ref{assum:field_h_distr} holds.\\
    1. Let $H$ be a $L$-Lipschitz field, i.e. for any $\bz_1, \bz_2\in Z_m$ it holds $\norm{H(\bz_1) - H(\bz_2)}\leq L\norm{\bz_1 - \bz_2}$. As shown in \cite{Lan_MPS}, Assumption \ref{assum:field_h_distr} holds with $M = L$ and $\delta = 0$.\\
    2. Let $H$ be bounded, i.e. for any $\bz\in Z_m$ it holds $\norm{H(\bz)}^2\leq C$. Then for any $\gamma > 0$ we have
    $$\langle H(\bz_1)-H(\bz_2), \bz_1-\bz_3\rangle \leq \frac{\gamma}{2}\| \bz_1- \bz_3\|^2 + \frac{2C}{\gamma},$$
    i.e. Assumption~\ref{assum:field_h_distr} holds with $M = \gamma$ and $\delta = 2C/\gamma$.

    The most interesting scenario from practical point of view is the latter case of bounded field $H$, since it corresponds to objective function with bounded subgradients. However, we still formulate Assumption~\ref{assum:field_h_distr} in a more general form of inequality \eqref{delta_L_H_bz}. The corresponding Theorems \ref{theorem_mps_deterministic_distr} and \ref{theorem_mps_stochastic_distr} are formulated and proved under general condition \eqref{delta_L_H_bz}, and the Corollaries \ref{corollary_mps_deterministic_distr} and \ref{corollary_mps_stochastic_distr} are derived for the specific case of bounded field norm.
\end{remark}

We introduce the following metric $Q(\bz_1, \bz_2) = G(\bz_1)-G(\bz_2) + \langle H(\bz_2), \bz_1-\bz_2 \rangle$. 
Note, that because of convexity $G$ we have the following estimation:
$$Q(\bz_1, \bz_2) = G(\bz_1)-G(\bz_2) + \langle H(\bz_2), \bz_1-\bz_2 \rangle\geq \langle \nabla G(\bz_2) + H(\bz_2), \bz_1-\bz_2 \rangle.$$
In other words, $Q(\bz_1, \bz_2)$ is estimation above for right hand of variational inequality. If for some $z_1$ we have $Q(\bz_1,\bz)\leq 0, \forall \bz\in Z_m$ then because of convexity $G$ we have that $\bz_1$ is solution of variational inequality \eqref{vi_problem_bz}.

In this section, we consider two methods based on the mirror prox-sliding from \cite{Lan_MPS} for problem \eqref{eq:spp_penalty}. In particular, we demonstrate that "the error" $\delta$ does not accumulate in sliding method.

\vspace{0.5cm}
\noindent\textbf{{Deterministic case}}.

The first method is the mirror prox-sliding without stochastic operator (see Algorithm \ref{alg:mps_deterministic_distr}). It contains two loops: external and internal. Note, that there is only one operation (line \ref{line:mps_determ_decentr_nabla_g}) where we need to communicate between nodes. At the same time, all other operations, including solving sub-problems and updating variables can be implemented locally at the nodes. So, the total number of communication equals $\sum\limits_{k=1}^n T_k.$

\begin{algorithm}[htp]
	\caption{The mirror-prox sliding (MPS) method for problem \eqref{eq:spp_penalty}} \label{alg:mps_deterministic_distr}
	\begin{algorithmic}[1]
		\REQUIRE Initial point $\bz_0 \in Z_m$, external steps number $N$,  sequence on internal steps number $\{T^k\}_{k=1}^n$, sequences of parameters $\{\beta_k\}_{k=1}^n, \{\eta_k^t\}_{t\leq T_k, k\leq N}, \{\gamma_k\}_{k=1}^N$
  
		\STATE $\overline{\bz}_0:=\bz_0$
		\FOR{$k =1, \ldots, N$}
        \STATE Compute $\underline{\bz}_k=(1-\gamma_k)\overline{\bz}_{k-1}+\gamma_k \bz_{k-1}$ and set $\bz_k^0 = \bz_{k-1}.$
    		\FOR{$t =1, \ldots, T_k$}
                    \STATE Compute $\underline{\mathbf{G}}_k:=\nabla G(\underline{\bz}_k)$ through one communication round \label{line:mps_determ_decentr_nabla_g}
                    \STATE Solve the following problems independently in $m$ nodes for $i=\overline{1,m}$:
                    $$\tilde{z}_{i,k}^t = \argmin\limits_{z\in Z} \langle \nabla \underline{{G}}_{i,k}+H_i(z_{i,k}^{t-1}), z\rangle + \beta_k V(z_{i, k-1}, z) + \eta_k^t V(z_{i,k}^{t-1}, z),$$        
                    $$z_{i,k}^t = \argmin\limits_{z\in Z} \langle \nabla \underline{{G}}_{i,k}+H_i(\tilde{z}_{i,k}^{t-1}), z\rangle + \beta_k V(z_{i, k-1}, z) + \eta_k^t V(z_{i,k}^{t-1}, z).$$
    		\ENDFOR 
      \STATE Set $\bz_k:=\bz_k^{T_k}, \tilde{\bz}_k:=\frac{1}{T_k}\sum\limits_{t=1}^{T_k} \tilde{\bz}_k^t$ and compute $\overline{\bz}_k=(1-\gamma_k)\overline{\bz}_{k-1}+\gamma_k\tilde{\bz}_k.$
		\ENDFOR 

        \RETURN $\overline{\bz}_N.$
	\end{algorithmic}
\end{algorithm}

Further, we assume that regularization coefficients $R_\alpha, R_\beta$ are defined according to \eqref{R_alpha_beta_det}. We also introduce the following values:
$$\chi_{\bullet}=\frac{\lambda_{\max} (\mW_{\bullet}^\top \mW_{\bullet})}{\lambda_{\min}^+ (\mW_{\bullet}^\top \mW_{\bullet})},\quad \chi=\max(\chi_x,\chi_y), \quad D=\max(\|\nabla f(\bx^*)\|,\sup_{\bx\in X_m}\|\nabla F(\bx, \by^*(\bx))\|).$$
Using the result our result for this sliding in non-smooth case (see Section \ref{sec:MPS}) and results about penalization (see Section \ref{sec:spp_penalty}) we can obtain the following result about convergence.
\begin{theorem}
\label{theorem_mps_deterministic_distr}
    Let us assume that $\varepsilon \geq \delta.$ Suppose that parameters in the outer iterations of Algorithm \ref{alg:mps_deterministic_distr} are set to:
    $$\gamma_k=\frac{2}{k+1},\beta_k=\frac{2L}{k},T_k=\left\lceil\frac{kM}{L}\right\rceil, \eta_k^t=\beta_k(t-1)+\frac{LT_k}{k},$$
    where $L=\frac{2D^2}{\varepsilon} \chi.$
    In order to compute an approximate solution $\overline{\bz}_N$ such that
    \begin{equation}
    \label{quality_det}
        \begin{aligned}
        &\max_{\substack{\mW_y\by=\mathbf{0}\\ \by\in Y_m}} F(\overline{\bx}_N, \by) - \min_{\substack{\mW_x\bx=\mathbf{0}\\ \bx\in X_m}} F(\bx, \overline{\by}_N)\leq 2\varepsilon, \\
            &\|\mathbf{W}_x\overline{\bx}_N\|\leq \frac{2\varepsilon}{R_\alpha},\quad \|\mathbf{W}_y\overline{\by}_N\|\leq \frac{2\varepsilon}{R_\beta},
        \end{aligned}
    \end{equation}the number of communication round and gradient calculation of $F$ are bounded by
$$N_{\text{\# rounds}}=O\left({\frac{D\Omega_{z_0} }{\varepsilon}\sqrt{\chi}}\right),\quad N_{\nabla F}=O\left(\frac{M\Omega_{z_0}^2+D\Omega_{z_0}\sqrt{\chi}}{\varepsilon}\right),$$
respectively, where $\Omega_{z_0}^2=\sup_{z\in Z} V(z_0, z).$
\end{theorem}

\begin{proof}
Firstly, note that operator $\nabla G$ in problem \eqref{eq:spp_penalty} satisfies Lipschitz condition with constant $\frac{2D^2}{\varepsilon} \chi.$

    Further, we obtain, that under conditions of Theorem \ref{theorem_mps_deterministic_distr} according to Theorem \ref{theorem_mps_deterministic} we obtain the point $\overline{\bz}_N$ such that
    \begin{equation}
    \label{eq1_theorem}
        Q(\overline{\bz}_N, \bz) = G(\overline{\bz}_N)-G(\bz) + \langle H(\bz), \overline{\bz}_N-\bz\rangle\leq 2\varepsilon, \forall \bz\in Z_m.
    \end{equation}
By the definition of operator $H$ and convexity of functions $F(\cdot, \by)$ for all $\by\in Y_m$ and $-F(\bx, \cdot)$ for all $\bx\in X_m$ we obtain that the inequality \eqref{eq1_theorem} is equivalent to the following:
\begin{equation}
\label{F_R_quality}
    F_R(\overline{\bx}_N, \by)-F_R(\bx, \overline{\by}_N)\leq 2\varepsilon, \forall \bx \in X_m, \by \in Y_m,
\end{equation}
where $F_R$ is defined in \eqref{eq:spp_penalty}. 

To conclude, the points $\overline{\bx}_N$ and $\overline{\by}_N$ satisfies condition of Lemma \ref{lem:spp_penalty} and it finalizes proof of this Theorem.

%
%
\end{proof}

Let us note $L_0 = \max_{\bz \in Z_m} \|H(\bz)\|.$ Then we can obtain that the following inequality holds:
\begin{equation}
\label{delta_L_H_bz_M}
    \langle H(\bz_1)-H(\bz_2), \bz_1-\bz_3\rangle \leq \frac{M}{2}\|\bz_1-\bz_3\|^2+\frac{2L_0^2}{M}, \forall \bz_1, \bz_2, \bz_3\in Z_m,
\end{equation}
for any positive constant $M>0.$ Note, in this case we have that $R\leq L_0$. So, taken $M=\frac{L_0^2}{2\varepsilon}$ in inequality \eqref{delta_L_H_bz_M} we can obtain from Theorem \eqref{alg:mps_deterministic_distr} the following result.
\begin{corollary}
\label{corollary_mps_deterministic_distr}
    Let us assume that $\|H(\bz)\|\leq L_0, \forall \bz\in Z_m.$ Suppose that parameters in the outer iterations of Algorithm \ref{alg:mps_deterministic_distr} are set as in Theorem \ref{theorem_mps_deterministic_distr}.
    In order to compute an approximate solution $\overline{\bz}_N$ such that \eqref{quality_det} holds, the number of communication round and gradient calculation of $F$ are bounded by
$$N_{\text{\# rounds}}=O\left({\frac{L_0\Omega_{z_0}}{\varepsilon}} \sqrt{\chi}\right),\quad N_{\nabla F}=O\left(\frac{L_0^2\Omega_{z_0}^2}{\varepsilon^2}+\frac{L_0\Omega_{z_0}\sqrt{ \chi}}{\varepsilon}\right),$$
respectively, where $\Omega_{z_0}^2=\sup_{z\in Z} V(z_0, z).$
\end{corollary}

So, we construct optimal deterministic algorithm for gradient calculations and number of communications. The convergence significantly depends on the size of the set $Z=X\times Y$ and network condition number $\chi$.

\vspace{0.5cm}
\noindent\textbf{{Stochastic case}}.

Further, we can consider the generalization of deterministic Algorithm~\ref{alg:mps_deterministic_distr} to stochastic case - Stochastic Mirror-Prox sliding (see Algorithm \ref{alg:mps_stochastic_distr}).

\begin{algorithm}[htp]
	\caption{The stochastic mirror-prox sliding (SMPS) method for problem \eqref{eq:spp_penalty}.} \label{alg:mps_stochastic_distr}
	\begin{algorithmic}[1]
		\REQUIRE Initial point $\bz_0 \in Z_m$, external steps number $N$,  sequence on internal steps number $\{T^k\}_{k=1}^n$, sequences of parameters $\{\beta_k\}_{k=1}^n, \{\eta_k^t\}_{t\leq T_k, k\leq N}, \{\gamma_k\}_{k=1}^N$
  
		\STATE Modify step 6 of Algorithm \ref{alg:mps_deterministic_distr} in the following way:
$$\tilde{z}_{i,k}^t = \argmin\limits_{z\in Z} \langle \nabla \underline{{G}}_{i,k}+\mathcal{H}_i(z_{i,k}^{t-1}; \zeta_{i,k}^{2t}), z\rangle + \beta_k V(z_{i, k-1}, z) + \eta_k^t V(z_{i,k}^{t-1}, z),$$        
$$z_{i,k}^t = \argmin\limits_{z\in Z} \langle \nabla \underline{{G}}_{i,k}+\mathcal{H}_i(\tilde{z}_{i,k}^{t-1}: \zeta_{i,k}^{2t}), z\rangle + \beta_k V(z_{i, k-1}, z) + \eta_k^t V(z_{i,k}^{t-1}, z).$$
	\end{algorithmic}
\end{algorithm}

In this case we consider the stochastic operator $\mathcal{H}.$ We will assume that the following assumption holds:
\begin{assumption}
\label{assump:h_i}
    Suppose that $\zeta_{i,k}^s$'s are independently random samples and the stochastic operator $\mathcal{H}$ satisfies unbiasedness $\mathbb{E}_{\zeta_{k}^{2t-1}} \left[\mathcal{H}\left(\bz_{k}^{t-1}; \zeta_{k}^{2t-1}\right)\right]=H(\bz_k^{t-1})$ and $\mathbb{E}_{\zeta_{k}^{2t}} \left[\mathcal{H}\left(\tilde{\bz}_k^t; \zeta_{k}^{2t}\right)\right]=H(\tilde{\bz}_k^t)$ and bounded variance $\mathbb{E}_{\zeta_{k}^{2t-1}} \left\|\mathcal{H}\left(\bz_{k}^{t-1}; \zeta_{k}^{2t-1}\right)-H(\bz_k^{t-1})\right\|^2\leq \sigma^2$ and $\mathbb{E}_{\zeta_{k}^{2t}} \left\|\mathcal{H}\left(\tilde{\bz}_k^t; \zeta_{k}^{2t}\right)-H(\tilde{\bz}_k^t)\right\|^2\leq \sigma^2$ for all $t=\overline{1, T_k}, k=\overline{1, N}$. Here $\zeta_{k}^s$ is the set of random samples $\zeta_{i,k}^s, i=\overline{1,m}$
\end{assumption}

Under this assumption, we can obtain the following result.

\begin{theorem}
\label{theorem_mps_stochastic_distr}
    Suppose that Assumption \ref{assump:h_i} holds and parameters in the outer iterations of Algorithm \ref{alg:mps_stochastic} are set to:
    $$\gamma_k=\frac{2}{k+1},\beta_k=\frac{2L}{k},T_k=\left\lceil\frac{\sqrt{3}kM}{L}+\frac{N k^2\sigma^2}{\Omega_{z_0}^2L^2}\right\rceil, \eta_k^t=\beta_k(t-1)+\frac{LT_k}{k}.$$
    where $L=\frac{2D^2}{\varepsilon} \chi, M=\frac{L_0^2}{2\varepsilon}.$
    In order to compute an approximate solution $\overline{\bz}_N$ such that the following holds with probability $1-p$:
    \begin{equation}
        \label{quality_stoch}
        \begin{aligned}
        &\max_{\substack{\mW_y\by=\mathbf{0}\\ \by\in Y_m}} F(\overline{\bx}_N, \by) - \min_{\substack{\mW_x\bx=\mathbf{0}\\ \bx\in X_m}} F(\bx, \overline{
    \by}_N)\leq 2\varepsilon, \\
            &\|\mathbf{W}_x\overline{\bx}_N\|\leq \frac{2\varepsilon}{R_\alpha},\quad \|\mathbf{W}_y\overline{\by}_N\|\leq \frac{2\varepsilon}{R_\beta},
        \end{aligned}
    \end{equation}the number of communication round and gradient calculation of $F$ are bounded by
$$N_{\text{\# rounds}}=O\left({\frac{D\Omega_{z_0}}{\varepsilon p}\sqrt{\chi}}\right),\quad N_{\nabla F}=O\left(\frac{M\Omega_{z_0}^2+D\Omega_{z_0}\sqrt{\chi}}{\varepsilon p} + \frac{\sigma^2 \Omega_{z_0}^2}{\varepsilon^2}\right),$$
respectively, where $\Omega_{z_0}^2=\sup_{z\in Z} V(z_0, z).$
\end{theorem}

\begin{proof}
    We have that if Algorithm \ref{alg:mps_stochastic} works with parameters from theorem then the output point $\overline{\bz}_N$ has quality $\mathbb{E}\sup_{\bz\in Z_n} Q(\overline{\bz}_N, \bz)\leq p\varepsilon.$ So, with according to Markov's inequality the point $\overline{\bz}_n$ is such that $\sup_{\bz\in Z_n} Q(\overline{\bz}_N, \bz)\leq \varepsilon$ with probability $1-p.$ Further, the proof repeats proof of Theorem \ref{theorem_mps_deterministic_distr}.  
\end{proof}

Similarly, we can obtain the following corollary.

\begin{corollary}
\label{corollary_mps_stochastic_distr}
    Let us assume that $\|H(\bz)\|\leq L_0, \forall \bz\in Z_m.$  Suppose that Assumption \ref{assump:h_i} holds and parameters in the outer iterations of Algorithm \ref{alg:mps_stochastic_distr} are set as in Theorem \ref{theorem_mps_stochastic_distr}.
    In order to compute an approximate solution $\overline{\bz}_N$ such that \eqref{quality_stoch} holds with probability $1-p$, the number of communication round and gradient calculation of $F$ are bounded by
$$N_{\text{\# rounds}}=O\left({\frac{L_0\Omega_{z_0}}{\varepsilon p}}\sqrt{\chi}\right),\quad N_{\nabla F}=O\left(\frac{L_0^2\Omega_{z_0}^2}{\varepsilon^2 p}+\frac{\sigma^2 \Omega_{z_0}^2}{\varepsilon^2 p}+\frac{L_0\Omega_{z_0}\sqrt{\chi}}{\varepsilon p}\right),$$
respectively, where $\Omega_{z_0}^2=\sup_{z\in Z} V(z_0, z).$
\end{corollary}

We can see that stochastic mirror-prox sliding does not accumulate the error $\delta$ for non-smooth function $F$ and variance $\sigma$ because of well-defined parameters.

\section{Discussion}\label{sec:discussion}

In this work, it was shown that the proposed Mirror-prox sliding framework can be applied to solving distributed non-smooth problems. This was achieved by replacing consensus restrictions with corresponding penalty functions. However, we note that in practice the approach of transition to dual constraints is used \cite{dual_approach_1}. Thus, the internal maximization subproblem can be replaced by a function. Note that in this way we can get rid of all consensus restrictions and reduce the saddle point problem to a minimization problem. However, the duality approach has a significant drawback. Namely, it is required to calculate the values of dual functions quite efficiently, and strong duality must also hold. The penalization approach proposed in this article does not have such disadvantages.

Another advantage of the proposed approach is the use of Bregman divergence. This may allow tuning for different geometry of the optimization problem (see \cite{bregman_orig}, \cite{bregman_example}).

However, the proposed approach has a typical problem. To run it, an estimate for the constants $L$ and $M,$ is required, as well as adjustment for the regularization parameters $R_{\alpha}, R_{\beta}.$ In addition, in the stochastic case, to achieve the required accuracy, an estimate of the gradient variance $\sigma^2$ and set size $\Omega_{z_0}$.

\section{Conclusion}\label{sec:conclusion}

In this work, we propose the optimal algorithm for decentralized saddle-point problem with non-smooth convex-concave object function. The main idea of the optimal approach is to use recently proposed in the work \cite{Lan_MPS} mirror-prox sliding for penalized initial problem with affine constraints. 

Firstly, proposed in work \cite{Lan_MPS} mirror-prox sliding is proven to converge to solution for non-smooth operator $H$. We proposed the generalization of Lipschitz condition for operator $H$ that depend on new parameter $\delta$. We obtain that the error $\delta$ in inexact setup \eqref{delta_L_H} does not accumulate by sliding. It means that the mirror-prox sliding method can be applied to a much wider class of optimization problem. 

Secondly, we provide analysis of well-known penalization technique for saddle-point problem. We obtain the conditions for penalization parameters to solve the initial problem with given accuracy to function value and to consensus constraints.

Finally, the proposed approach was considered in both deterministic and stochastic cases. We show that in both cases it finds $\varepsilon$-optimal point through $O\left(\frac{\sqrt{\chi}}{\varepsilon}\right)$ communications steps totally and $O\left(\frac{1}{\varepsilon^2}\right)$ gradient calls per node. Note, these estimations are optimal and can not be improved.

Thus, this work contains optimal distributed algorithm for enough wide and complex class of saddle-point problems without Lipschitz gradient and without conditions of strong convexity or strong concavity.

\bibliography{biblio_new}

\end{document}